\documentclass[12pt]{amsart}
\usepackage{amssymb}
\usepackage{amsmath}
\usepackage{amsthm}

\newcommand{\e}{\varepsilon}

\DeclareMathOperator{\iindex}{index}
\DeclareMathOperator{\sign}{sign}
\DeclareMathOperator{\meas}{meas}

\DeclareMathOperator{\supp}{supp}
\DeclareMathOperator{\rank}{rank}
\DeclareMathOperator{\Ran}{Ran}
\DeclareMathOperator{\Dom}{Dom}
\DeclareMathOperator{\dom}{dom}
\DeclareMathOperator{\Ker}{Ker}
\DeclareMathOperator{\Tr}{Tr}

\renewcommand\Re{\hbox{{\rm Re}}\,}
\newcommand{\abs}[1]{\lvert#1\rvert}
\newcommand{\aabs}[1]{\left\lvert#1\right\rvert}
\newcommand{\norm}[1]{\lVert#1\rVert}

\newcommand{\R}{{\mathbb R}}

\newcommand{\C}{{\mathbb C}}

\newcommand{\HH}{{\mathcal H}}
\newcommand{\K}{{\mathcal K}}

\newcommand{\frakS}{{\mathfrak S}}

\numberwithin{equation}{section}

\theoremstyle{plain}
\newtheorem{theorem}{\bf Theorem}[section]
\newtheorem{lemma}[theorem]{\bf Lemma}
\newtheorem{proposition}[theorem]{\bf Proposition}

\newtheorem{corollary}[theorem]{\bf Corollary}

\theoremstyle{definition}

\theoremstyle{remark}
\newtheorem*{remark*}{\bf Remark}

\renewcommand{\qed}{\vrule height7pt width5pt depth0pt}
\newcommand{\T}{{\mathbb T}}

\let\geq\geqslant
\let\leq\leqslant

\DeclareMathOperator{\ind}{index}
\newcommand{\wt}{\widetilde}
\DeclareMathOperator{\sflow}{sf}

\begin{document}
\title[Eigenvalues in the gaps]{Operator theoretic methods for the eigenvalue counting function in spectral gaps}
\sloppy

\author{Alexander Pushnitski}
\address{Department of Mathematics\\
King's College London\\ 
Strand, London WC2R  2LS, U.K.}
\curraddr{}
\email{alexander.pushnitski@kcl.ac.uk}
\thanks{}
\date{25 September 2008}

\subjclass[2000]{Primary 35P20; Secondary 47B25, 47F05}

\keywords{spectral flow, spectral gaps, Birman-Schwinger principle, spectral shift function, Landau levels}

\begin{abstract}
Using the notion of spectral flow, we suggest a simple approach to 
various asymptotic problems involving eigenvalues in the gaps of 
the essential spectrum of self-adjoint operators. 
Our approach uses some elements of the spectral shift function 
theory. Using this approach, we provide generalisations and 
streamlined proofs of two results in this area already existing in 
the literature. We also give a new proof of the generalised 
Birman-Schwinger principle. 
\end{abstract}

\maketitle

\section{Introduction}\label{sec.a}

\subsection{The spectral flow}
Since the pioneering work \cite{ADH}, problems involving counting functions of eigenvalues
in the gaps of the essential spectrum of self-adjoint operators attracted a  considerable amount 
of attention in the mathematical physics literature.
Let us recall the set-up of the problem. 
Let $M$ and $A$ be self-adjoint operators in a Hilbert space 
such that the spectrum of $M$ has a gap and $A$ is $M$-compact. 
Then, for any $t\in\R$, the essential spectra of $M$ and $M+tA$ coincide 
and the eigenvalues of  $M+tA$ in the spectral gaps of $M$ are analytic in $t$. 
If $A\geq0$ or $A\leq0$ in the quadratic form sense, then these eigenvalues are monotone 
in $t$; in general, they may not be monotone. 

Let us fix  a coupling constant $t>0$  and  a spectral parameter $\lambda$ in a spectral gap of $M$ and consider 
one of the variants of the eigenvalue counting function, known as 
the spectral flow of the family $M+\tau A$, $\tau \in[0,t]$, through $\lambda$. 
This is defined as follows. As $\tau$ increases monotonically from $0$ to $t$, 
some eigenvalues of $M+\tau A$ may cross $\lambda$. By analyticity in $\tau$, there
will only be  a finite number of such crossings. Some eigenvalues will cross 
$\lambda$ from left to right, others from right to left. 
The spectral flow is defined as  
\begin{multline}
\sflow(\lambda; M+t A,M)
\\
=
\langle 
\text{the number of eigenvalues of $M+\tau A$, $0\leq \tau \leq t$, which cross $\lambda$ rightwards}
\rangle
\\
-
\langle 
\text{the number of eigenvalues of $M+\tau A$, $0\leq \tau\leq t$, which cross $\lambda$ leftwards}
\rangle.
\label{a0}
\end{multline}
Some eigenvalues may ``turn around'' at $\lambda$ (i.e. for some $\tau_0\in(0,t)$, the function $\lambda_n(\tau)$ 
may have a local minimum or local maximum at $\tau=\tau_0$); these eigenvalues do not contribute to \eqref{a0}. 
The eigenvalues are counted with multiplicities taken into account.

The asymptotics of $\sflow(\lambda; M+tA,M)$ as $t\to\infty$ and related issues
have been extensively studied both for concrete differential operators 
$M+tA$ and in an abstract setting; see e.g. the survey 
\cite{Hempel2} for the history and a recent paper \cite{Simon100DM}
for extensive bibliography. 
Most relevant to our approach are the operator theoretic constructions of 
M.~Birman (see \cite{Birman} and references therein)
 and O.~Safronov \cite{Safronov1,Safronov2,Safronov3}.
We also note that there is a large family of index theorems 
(see e.g. \cite{RobbinSalamon} and references therein) 
which use the notion of the spectral flow; these are not directly 
related to the topic of this paper.

\subsection{Spectral flow, Fredholm index, and spectral shift function}
Let us start by mentioning two other interpretations of the spectral flow; 
the precise statements will be given in Section~\ref{sec.b}.
First, if $A$ is a trace class operator, then 
\begin{equation}
\sflow(\lambda; M+A,M)=\xi(\lambda-0; M+A,M),
\quad \forall \lambda\in\R\setminus\sigma_{ess}(M),
\label{b17}
\end{equation}
where $\xi(\cdot; M+A,M)$ is the  M.~Krein's spectral shift function.
Next, one has
$$
\sflow(\lambda; M+A,M)=\Xi(\lambda;M+A,M),
$$
where the right hand side is defined as the Fredholm index 
(see Section~\ref{sec.b1})
of the 
pair of spectral projections of $M+A$ and $M$, associated 
with the interval $(-\infty,\lambda)$: 
\begin{equation}
\Xi(\lambda; M+A,M)=\iindex(E_{M}(-\infty,\lambda),E_{M+A}(-\infty,\lambda)).
\label{a7}
\end{equation}
These interpretations of spectral flow have now become folklore; 
they have also been used in the abstract
operator theoretic context, in particular in the works on 
operator algebras, see e.g. \cite{CareyPincus,ACS} and references therein.
However, the methods emerging from these interpretations 
have not yet been used to the full extent in the mathematical physics 
literature. This paper aims to fill in this gap.

We consider the function $\Xi$ defined by \eqref{a7}; the precise 
definition is given in Section~\ref{sec.b}. 
We use the intuition coming from the spectral shift function theory 
to provide elementary proofs of a number of simple yet very useful properties of
this function. 
Most importantly, one has the ``chain rule''
\begin{equation}
\Xi(\lambda; M+A_1+A_2,M)
=
\Xi(\lambda; M+A_1,M)+\Xi(\lambda; M+A_1+A_2,M+A_1)
\label{a8}
\end{equation} and the estimates 
(see Theorem~\ref{th.b1})
\begin{equation}
-\rank A_-\leq \Xi(\lambda; M+A,M)\leq \rank A_+,
\qquad 
A_\pm=\frac12(\abs{A}\pm A).
\label{b12}
\end{equation}
In particular, 
\begin{equation}
\pm A\geq 0
\quad \Rightarrow\quad
\pm \Xi(\lambda; M+A,M)\geq 0.
\label{b13}
\end{equation}
These properties are well known in the spectral shift function theory.
In our approach, they 
provide a basis for various monotonicity arguments 
typical for variational technique. 

Next, the function $\Xi$ is related to the eigenvalue counting function by 
\begin{equation}
\Xi(\lambda_1; M+A,M)-\Xi(\lambda_2;M+A,M)
=
N([\lambda_1,\lambda_2);M+A)
-
N([\lambda_1,\lambda_2);M);
\label{b20}
\end{equation}
here $N(\delta; M)$ is the number of eigenvalues of $M$ 
in the interval $\delta$, and we assume that
$\sigma_{ess}(M)\cap [\lambda_1,\lambda_2]=\varnothing$.
A particular case of \eqref{b20} is 
\begin{equation}
\Xi(\lambda; M+A,M)=-N((-\infty,\lambda); M+A), \quad  \lambda<\inf\sigma(M).
\label{b21}
\end{equation}
Because of these properties, $\Xi$ is a useful tool in analysing 
the eigenvalue counting function in the gaps of essential spectrum. 

Further, in Section~\ref{sec.b4} we consider the behaviour of $\Xi$ 
with respect to decompositions
of the Hilbert space into direct sums.
Finally, in Section~\ref{sec.g},
we discuss and provide a new proof of the identity which can 
be interpreted as the Birman-Schwinger principle stated in terms of $\Xi$.

Most of these properties of $\Xi$ appeared before in the literature in various guises,
see e.g. \cite{GMN,GM,BPR,Simon100DM}, 
mainly (but not exclusively) in the framework of the spectral shift function theory, 
which requires some trace class assumptions.
The novelty of this paper is in collecting these properties together 
in a unified and rather general form  and putting them to work in 
problems involving eigenvalues in the gaps outside the trace class scheme.
We also provide streamlined and self-contained proofs of these properties. 

We do not make any attempt here to review the literature on the eigenvalue
counting function, as it it enormously wide. 
Where appropriate, we only mention the works most directly 
related to our approach. 
More references and history can be found in the survey \cite{Hempel2} 
and the recent paper \cite{Simon100DM}.
We also note that an interesting approach to the analysis of eigenvalues in the spectral 
gaps has been developed in \cite{Siedentop,Esteban}. It doesn't seem to be directly 
related to the approach of this paper. Some discussion of the numerical aspects of
calculation of eigenvalues in the gaps and appropriate references
can be found e.g. in \cite{Davies}.

\subsection{Applications}
To illustrate the efficiency of our approach, we apply it 
to provide simple proofs of two results already present in the literature. 
The first one is a theorem of O.~Safronov from \cite{Safronov2} which deals with 
the asymptotics of $\Xi(\lambda; M+tA,M)$ as $t\to\infty$. 
A typical application of this theorem is to the spectral flow of the Schr\"odinger
operator $M$ with a periodic potential, perturbed by the 
operator $A$ of multiplication by a potential which decays at infinity. 
This is discussed in Section~\ref{sec.d}. 

The second result is a theorem of G.~Rozenblum and A.~Sobolev \cite{RS}
which describes the asymptotic eigenvalue distribution of the Landau operator
perturbed by an expanding potential. 
In Sections~\ref{sec.e} and \ref{sec.c}, 
we provide a streamlined proof and a generalisation of this result.

In conclusion, we list other potential areas of application of our technique:

(i) Theorem~\ref{th.e1} can be applied to the analysis of a periodic operator perturbed by 
an expanding potential. 

(ii) 
Analysis of eigenvalues in the gap of the Dirac operator. 
This will require a generalisation of our technique to the case of the operators
which are not lower bounded. 

\subsection{Notation}
For a self-adjoint operator $A$,  the symbols
 $\sigma(A)$, $\sigma_{ess}(A)$,  $E_A(a,b)$, and $N(\delta; A)$ denote the 
 spectrum of $A$, the essential spectrum of $A$,
the  spectral projection of $A$ associated with $(a,b)\subset\R$,
and the total number of eigenvalues (counting multiplicity) of $A$ in the 
interval $\delta$.  
The symbols $\frakS_\infty$ and $\frakS_2$ denote the classes 
of compact and Hilbert-Schmidt operators in a Hilbert space.

\section{The function $\Xi$ }\label{sec.b}
In this section we introduce the function $\Xi$ and discuss its relationship with 
the spectral shift function, the spectral flow and the eigenvalue counting function. 
We discuss the stability of $\Xi$ and prove variational estimates which will be 
crucial for our further analysis.
We also discuss the behaviour of $\Xi$ with respect to 
the decomposition of the Hilbert space into direct sums. 
\subsection{The index of a pair of projections}\label{sec.b1}
Let us recall some background material from \cite{ass}.
A pair of orthogonal projections $P$, $Q$ in a Hilbert space $\HH$
is called Fredholm, if
$$
\{1,-1\}\cap\sigma_{\rm ess}(P-Q)=\varnothing.
$$
In particular, if $P-Q$ is compact, then the pair $P$, $Q$ is Fredholm.
The index of a Fredholm pair is given by the formula
$$
\ind(P,Q)=\dim\Ker(P-Q-I)-\dim\Ker(P-Q+I).
$$
This can be alternatively written as
\begin{equation}
\ind(P,Q)=\dim(\Ran P\cap \Ker Q)-\dim(\Ran Q\cap\Ker P).
\label{b00}
\end{equation}
It is well known (see e.g. \cite[Theorem~4.2]{ass}) that 
$$
\dim \Ker(P-Q-\lambda I)=\dim \Ker(P-Q+\lambda I), 
\quad \lambda\not=\pm 1;
$$
the proof of this is based on the identity
$$
(P-Q)(I-P-Q)=(I-P-Q)(Q-P).
$$
Thus, if $P-Q$ is a trace class operator,  then 
\begin{equation}
\ind(P,Q)=\Tr(P-Q),
\label{b0}
\end{equation}
since all the eigenvalues of $P-Q$ apart from $1$ and $-1$  in the series $\Tr(P-Q)=\sum_k\lambda_k(P-Q)$
cancel out.

If both $(P,Q)$ and $(Q,R)$ are Fredholm pairs
and at least one of the differences $P-Q$ or $Q-R$ is compact, then the pair $(P,R)$
is also Fredholm and the following identity holds true:
\begin{equation}
\label{b1}
\ind(P,R)=\ind(P,Q)+\ind(Q,R).
\end{equation}
See e.g. \cite{ass} for the proof of the last statement and the details.

\subsection{Definition of $\Xi$}\label{sec.b1a}
Let $M$ and $\wt M$ be self-adjoint operators in a Hilbert space $\HH$. 
If $E_M(-\infty,\lambda)$, $E_{\wt M}(-\infty,\lambda)$ is a Fredholm pair, we will 
say that $\Xi(\lambda;\wt M, M)$ exists and define
$$
\Xi(\lambda; \wt M,M):=\ind\bigl(E_{M}(-\infty,\lambda),E_{\wt M}(-\infty,\lambda)\bigr).
$$
It is obvious that $\Xi(\lambda; \wt M,M)$
 is constant on the intervals 
of the set $\R\setminus(\sigma(M)\cup\sigma(\wt M))$ and 
 is left continuous in $\lambda$
on the set $\R\setminus(\sigma_{ess}(M)\cup\sigma_{ess}(\wt M))$.
One has
$$
\Xi(\lambda; \wt M,M)=-\Xi(\lambda; M,\wt M).
$$

We will often use the following simple sufficient condition for the existence
of $\Xi(\lambda; \wt M,M)$. 
Let  $M$ be a self-adjoint lower semi-bounded  operator in $\HH$  and let 
the self-adjoint operator  $A$ in $\HH$ be relatively form-compact
with respect to $M$.
This means that 
$$
\Dom \abs{A}^{1/2}\supset \Dom(M+\gamma I)^{1/2}
\quad
\text{ and }\quad
\abs{A}^{1/2}(M+\gamma I)^{-1/2}\in {\frakS_\infty}
$$
for all sufficiently large $\gamma>0$.
Under this assumption, by the KLMN Theorem (see \cite[Theorem~X.17]{RS2})
the operator $\wt M=M+A$ is well defined in terms of the corresponding  quadratic form.
Using the resolvent identity, we get
\begin{equation}
(\wt M-zI)^{-1}-(M-zI)^{-1}\in {\frakS_\infty}
\quad \forall z\in\C\setminus(\sigma(M)\cup\sigma(\wt M)).
\label{b11}
\end{equation}
By Weyl's theorem on the stability of essential spectrum under compact perturbations,
this implies $\sigma_{ess}(\wt M)=\sigma_{ess}(M)$. 
If $\lambda\in\R\setminus\sigma_{ess}(M)$, then, representing the spectral projections by 
Riesz integrals and using \eqref{b11}, it is easy to see that the difference 
$E_{\wt M}(-\infty,\lambda)-E_M(-\infty,\lambda)$ is compact and therefore
$\Xi(\lambda;\wt M,M)$ exists. 

If both $A_1$ and $A_2$ are form-compact with respect to $M$ and 
$\lambda$ is not in the essential spectrum of $M$, then, by the above 
argument and \eqref{b1}, the ``chain rule'' \eqref{a8} holds true.

\begin{remark*}
In this paper, we  assume most of the time that $M$ is
lower semi-bounded. It is possible to generalise our results to the case of 
non-semibounded $M$. However, this makes our construction and particularly the proofs considerably
more complicated. 
\end{remark*}

The function $\Xi$, with various notation and in various guises, 
appeared in the literature many times. 
Without any attempts at being exhaustive, let us mention a few sources.
In \cite{GMN,GM,BPR}, $\Xi$ was used in the context of the spectral 
shift function theory. 
There is extensive literature on $\Xi$ in the theory of operator algebras,
see e.g. \cite{ACDS,ACS,KMS} and references therein.

\subsection{$\Xi$ and the spectral shift function}

\begin{proposition}\label{prp.ssf}
If $M$ and $A$ are self-adjoint operators and 
$A$ is a trace class operator, then the identity 
\begin{equation}
\Xi(\lambda; M+A,M)=\xi(\lambda-0; M+A,M),
\quad \forall \lambda\in\R\setminus\sigma_{ess}(M)
\label{b17a}
\end{equation}
holds true, 
where $\xi(\cdot; M+A,M)$ is  the M.~Krein's spectral shift function.
\end{proposition}
\begin{proof}
The key point is  the M.~Krein's theorem that if 
$\varphi'\in C_0^\infty(\R)$, then $\varphi(M+A)-\varphi(M)$ belongs to the 
trace class and the  trace formula
$$
\Tr(\varphi(M+A)-\varphi(M))
=
\int_{-\infty}^\infty \varphi'(\lambda)\xi(\lambda; M+A,M)d\lambda
$$
holds true. 
Since $\lambda$ is not in the essential spectrum of $M$, there exists $\delta>0$ 
such that $(\lambda-\delta,\lambda)\cap \sigma(M)=\varnothing$
and $(\lambda-\delta,\lambda)\cap \sigma(M+A)=\varnothing$.
Then we can choose $\varphi$ with $\supp \varphi'\subset(\lambda-\delta,\lambda)$ 
and $\varphi(\lambda-\delta)=1$, $\varphi(\lambda)=0$. 
Then, using \eqref{b0}, we get
\begin{multline*}
\Tr(\varphi(M+A)-\varphi(M))
=
\Tr(E_{M+A}(-\infty,\lambda)-E_M(-\infty,\lambda))
\\
=
\iindex(E_{M+A}(-\infty,\lambda),E_{M}(-\infty,\lambda))
=
-\Xi(\lambda; M+A,M).
\end{multline*}
On the other hand, since $\xi(\lambda; M+A,M)$ is constant 
on $(\lambda-\delta,\lambda)$, we get
$$
\int_{-\infty}^\infty \xi(\lambda; M+A,M)\varphi'(\lambda)d\lambda
=
-\xi(\lambda-0; M+A,M); 
$$
this proves the claim. 
\end{proof}
As mentioned in the introduction, this statement can be regarded as folklore; 
it was explicitly stated and used e.g. in \cite{Safronov3,ACS,KMS}.

\subsection{$\Xi$ and  the eigenvalue counting function.}
\begin{theorem}\label{th.b5}
Let $M$ be a lower semibounded self-adjoint operator, let $A$ be a
self-adjoint operator which is form-compact with respect to $M$
and $[\lambda_1,\lambda_2]\subset\R\setminus\sigma_{ess}(M)$.
Then the identities \eqref{b20}, \eqref{b21} hold true. 
\end{theorem}
\begin{proof} 
Using  \eqref{b1} and \eqref{b0}, we get
\begin{multline*}
\Xi(\lambda_1; M+A,M)
=
\ind(E_M(-\infty,\lambda_1),E_{M+A}(-\infty,\lambda_1))
\\
=
\ind(E_M(-\infty,\lambda_1),E_{M}(-\infty,\lambda_2))
+
\ind(E_M(-\infty,\lambda_2),E_{M+A}(-\infty,\lambda_2))
\\
+
\ind(E_{M+A}(-\infty,\lambda_2),E_{M+A}(-\infty,\lambda_1))
\\
=
\Tr(E_M(-\infty,\lambda_1)-E_{M}(-\infty,\lambda_2))
+
\Xi(\lambda_2; M+A,M)
\\
+
\Tr(E_{M+A}(-\infty,\lambda_2)-E_{M+A}(-\infty,\lambda_1))
\\
=
-N([\lambda_1,\lambda_2);M)+\Xi(\lambda_2; M+A,M)+N([\lambda_1,\lambda_2); M+A);
\end{multline*}
this proves \eqref{b20}. 
Identity \eqref{b21} follows by taking $\lambda_1\to-\infty$ and $\lambda_2=\lambda$. 
\end{proof}

This statement is well known and (if stated in terms of the spectral flow) is intuitively obvious. 
\subsection{Stability of $\Xi$}
The following result is essentially well known; see \cite[Theorem~3.12]{GM}
for a very similar statement. 
However, in order to make this text self-contained, we provide
a proof (which is not significantly different from the  proof of 
\cite{GM}).
\begin{theorem}\label{stab}
Let $M$ and $\wt M$ be lower semi-bounded self-adjoint operators 
and suppose that $\Xi(\lambda; \wt M,M)$ exists for some 
$\lambda\in\R\setminus(\sigma(M)\cup\sigma(\wt M))$.
Let $M_n$ and $\wt M_n$ be two sequences of self-adjoint operators such that
$M_n\to M$ and $\wt M_n\to \wt M$ in the norm resolvent sense and 
$M_n$, $\wt M_n$ are uniformly bounded from below:  
$\gamma I\leq M_n$, $\gamma I\leq \wt M_n$ for some $\gamma\in\R$ and all $n$. 
Then for all sufficiently large $n$, $\Xi(\lambda; \wt M_n, M_n)$ 
exists and equals $\Xi(\lambda; \wt M,M)$. 
\end{theorem}
\begin{remark*}
It is not difficult to construct an example showing that the assumption of 
the existence of a uniform lower bound 
for $M_n$ and $\wt M_n$ cannot be dropped
from the hypothesis of this theorem.
\end{remark*}
\begin{proof}
1. 
Let us denote 
$$
P=E_M(-\infty,\lambda), 
\quad 
P_n=E_{M_n}(-\infty,\lambda),
\quad
Q=E_{\wt M}(-\infty,\lambda), 
\quad
Q_n=E_{\wt M_n}(-\infty,\lambda). 
$$
By our assumptions, 
$P=E_M(\gamma-1,\lambda)$, $P_n=E_{M_n}(\gamma-1,\lambda)$.
Since $\lambda$ and $\gamma-1$ are not in the spectrum of $M$,
by \cite[Theorem VIII.23(b)]{RS1}
it follows that $\norm{P_n-P}\to0$ as $n\to\infty$. 
In the same way, $\norm{Q_n-Q}\to0$ as $n\to\infty$. Thus, 
\begin{equation}
\norm{ (P_n-Q_n)-(P-Q)}\to 0 \quad \text{ as $n\to\infty$.}
\label{b29}
\end{equation}

2. 
Since $(P,Q)$ is a Fredholm pair, there exists $\delta>0$ such that 
\begin{equation}
\sigma(P-Q)\cap(-1,1)\subset [-1+2\delta,1-2\delta].
\label{b30}
\end{equation}
Then $-1+\delta$ and $1-\delta$ are not in the spectrum of $P-Q$ 
and so, using 
\eqref{b29} and applying \cite[Theorem VIII.23(b)]{RS1}
again, we get
\begin{align}
\norm{E_{P_n-Q_n}(1-\delta,2)-E_{P-Q}(1-\delta,2)}&\to0,
\label{b31}
\\
\norm{E_{P_n-Q_n}(-2,-1+\delta)-E_{P-Q}(-2,-1+\delta)}&\to0.
\label{b32}
\end{align}
In particular, $\rank E_{P_n-Q_n}(1-\delta,2)$ and $\rank E_{P_n-Q_n}(-2,-1+\delta)$
are finite for large $n$ and therefore $\Xi(\lambda; \wt M_n,M_n)$ exists.

3. 
By the definition of index and \eqref{b30}, we have 
\begin{equation}
\Xi(\lambda; \wt M,M)
=
\rank E_{P-Q}(1-\delta,2)-\rank E_{P-Q}(-2,-1+\delta).
\label{b33}
\end{equation}
By \eqref{b29} and \eqref{b30}, we have 
$\sigma(P_n-Q_n)\cap(-1,1)\subset (-1+\delta,1-\delta)$
for all sufficiently large $n$, and then 
\begin{equation}
\Xi(\lambda; \wt M_n,M_n)
=
\rank E_{P_n-Q_n}(1-\delta,2)-\rank E_{P_n-Q_n}(-2,-1+\delta).
\label{b34}
\end{equation}
Combining \eqref{b31}--\eqref{b34}, we get the required 
statement. 
\end{proof}

\subsection{$\Xi$ as spectral flow}\label{sec.b3}
Let $M$ be a lower semi-bounded self-adjoint operator and let $A$ be 
form-compact with respect to $M$. Here we prove that 
\begin{equation}
\sflow(\lambda; M+A,M)=\Xi(\lambda; M+A,M), 
\quad
\lambda\in \R\setminus(\sigma(M)\cap \sigma(M+A)).
\label{a1}
\end{equation}
This statement is not used elsewhere in the paper and is
given here only in order to provide some motivation and help comparison with 
other results in the area.

Let us choose an interval $(a,b)\subset\R\setminus\sigma_{ess}(M)$ such that  $\lambda\in(a,b)$.
First suppose that there exists $\lambda_0<\lambda$, 
$\lambda_0\in(a,b)$ such that 
\begin{equation}
\lambda_0\notin\sigma(M+tA)
\quad \text{ for all $t\in[0,1]$.}
\label{b2}
\end{equation}
It is easy to see that the resolvent $(M+tA-zI)^{-1}$, $z\in \C\setminus\R$, 
is continuous in $t\in[0,1]$ in the operator norm.
By the stability Theorem~\ref{stab}, we conclude that $\Xi(\lambda_0; M+tA,M)$
is independent of  $t\in[0,1]$ and therefore  
$\Xi(\lambda_0; M+A,M)=0$. 
Using \eqref{b20}, we get
$$
\Xi(\lambda; M+A,M)=N([\lambda_0,\lambda); M)-N([\lambda_0,\lambda);M+A).
$$
The r.h.s. equals the net flux of eigenvalues of $M+tA$ outward from the 
interval $(\lambda_0,\lambda)$. By \eqref{b2}, the flux through $\lambda_0$ equals zero. 
Thus, it is clear that \eqref{a1} holds true. 

In general, the point $\lambda_0$ as in \eqref{b2} may not exist, but we can always 
find  a finite open cover of  $[0,1]$ by sufficiently 
small subintervals $\delta_i$ such that for each family $\{M_t\mid t\in \delta_i\}$, 
the point $\lambda_0$ can be chosen appropriately. Then formula \eqref{a1} can be obtained by 
combining the formulas corresponding to all the subintervals.

\subsection{Variational estimates for  $\Xi$}\label{sec.b2}
\begin{theorem}\label{th.b1}
Let $M$ be a  
lower semi-bounded self-adjoint operator and let the self-adjoint  operator 
$A$ be relatively form bounded with respect to $M$ with a relative bound less than one.
Let $\wt M$ be defined as a form sum $\wt M=M+A$; 
assume that $\Xi(\lambda; \wt M,M)$ exists. 
Then the estimates \eqref{b12} and \eqref{b13} hold true.
\end{theorem}
\begin{proof}
Let us prove the second inequality in \eqref{b12}; the proof of the first one 
is analogous. 
If $\rank A_+=\infty$, there is nothing to prove; so let us assume $\rank A_+<\infty$. 
By \eqref{b00},
the desired statement will follow if we prove that 
$$
\dim(\Ran E_{M}(-\infty,\lambda)\cap \Ker E_{\wt M}(-\infty,\lambda))
\leq \rank A_+.
$$
Suppose to the contrary that 
$$
\dim(\Ran E_{M}(-\infty,\lambda)\cap \Ker E_{\wt M}(-\infty,\lambda))
> \rank A_+.
$$
Then there exists $\psi\not=0$, $\psi\in (\Ran E_{M}(-\infty,\lambda)\cap \Ker E_{\wt M}(-\infty,\lambda))$
and $A_+\psi=0$. 
Denote by $m$ and $\wt m$ the sesquilinear forms corresponding to $M$ and $\wt M$. 
Since $\psi\in \Ran E_{M}(-\infty,\lambda)$,
it follows that $\psi\in\dom(m)=\dom(\wt m)$. We have 
\begin{equation}
m[\psi,\psi]<\lambda\norm{\psi}^2
\quad \text{ and }\quad 
\wt m[\psi,\psi]\geq \lambda \norm{\psi}^2.
\label{b16}
\end{equation}
On the other hand, since $A_+\psi=0$, we have $\wt m[\psi,\psi]\leq m[\psi,\psi]$, 
which is a contradiction with \eqref{b16}.
\end{proof}
This result immediately implies the following monotonicity principle, variants
of which have been used before, e.g. in \cite{Safronov3,Simon100DM}:

\begin{corollary}\label{cr.b2}
Let $M_1=M+A_1$, $M_2=M+A_2$, where the operator $M$ is self-adjoint and 
lower semibounded and 
the self-adjoint operators $A_1$ and $A_2$ are form compact with respect to $M$.  
Let $\lambda\in\R\setminus\sigma_{ess}(M)$ and suppose that $M+A_2\geq M+A_1$; then 
$$
\Xi(\lambda; M+A_2,M)\geq \Xi(\lambda; M+A_1,M).
$$
\end{corollary}
Note that  $M+A_2\geq M+A_1$ can usually be 
written in a simpler form $A_2\geq A_1$, but  this requires that
the quadratic forms corresponding to  
$A_2$ and $A_1$ are well defined. 

\begin{proof}
By the ``chain rule" \eqref{a8}, 
$$
\Xi(\lambda; M+A_2,M)=\Xi(\lambda; M+A_2,M+A_1)+\Xi(\lambda; M+A_1,M);
$$
by Theorem~\ref{th.b1}, the first term in the r.h.s. is non-negative. 
\end{proof}

A simple example of the application of this monotonicity principle is an estimate of the 
number of eigenvalues in the gap of $M$ 
when the perturbation $A$ can be represented as $A=B-C$ with $B\geq0$ and $C\geq0$.  
Indeed, if $[\lambda_1,\lambda_2]\cap\sigma(M)=\varnothing$, then by \eqref{b20} 
we have
\begin{multline*}
N([\lambda_1,\lambda_2); M+B-C)
=
\Xi(\lambda_1; M+B-C)-\Xi(\lambda_2; M+B-C)
\\
\leq
\Xi(\lambda_1; M+B,M)-\Xi(\lambda_2; M-C,M);
\end{multline*}
now the right hand side can be evaluated, for example, by 
using the Birman-Schwinger principle, see \eqref{b23} and \eqref{b24} below.
This argument has been used before, see e.g. \cite{Simon100DM}.

\begin{corollary}\label{cr.b3}
Let $M$ be a lower semibounded self-adjoint operator such that $[\lambda-a,\lambda+a]\cap\sigma(M)=\varnothing$ 
for some $\lambda\in \R$ and $a>0$. 
Let $A$ and $B$ be compact self-adjoint operators. Then 
\begin{align*}
\Xi(\lambda; M+A+B,M)
&\leq 
\Xi(\lambda-a; M+A,M)+N((a,\infty);B),
\\
\Xi(\lambda; M+A+B,M)
&\geq
\Xi(\lambda+a; M+A,M)-N((-\infty,-a);B).
\end{align*}
In particular, 
\begin{equation}
-N((-\infty,-a); B)
\leq 
\Xi(\lambda; M+B,M)
\leq 
N((a,\infty); B).
\label{a9}
\end{equation}
\end{corollary}
Note that \eqref{a9} is an improvement of \eqref{b12}, given some 
information on the width of the spectral gap of $M$ around $\lambda$. 
\begin{proof}
Let us write $B=B_1+B_2$, where $\norm{B_1}\leq a$ and 
$\rank(B_2)_+=N((a,\infty);B)$, $\rank(B_2)_-=N((-\infty,-a);B)$. 
Then, by \eqref{b1} and  Theorem~\ref{th.b1}, one has
\begin{multline*}
\Xi(\lambda; M+A+B,M)
=
\Xi(\lambda; M+A+B_1+B_2,M+A+B_1)
\\
+
\Xi(\lambda; M+A+a I +(B_1-aI),M+A+a I)
+
\Xi(\lambda; M+A+a I,M)
\\
\leq 
\rank(B_2)_+ 
+
\Xi(\lambda; M+A+a I,M)
\\
=
N((a,\infty);B)+\Xi(\lambda; M+A+a I,M+a I)
=
N((a,\infty);B)+\Xi(\lambda-a; M+A,M).
\end{multline*}
This proves the upper bound for $\Xi(\lambda; M+A+B,M)$; 
the lower bound is proven in an analogous way. 
\end{proof}

Theorem~\ref{th.b1} and Corollary~\ref{cr.b3} appeared before in 
\cite{BPR} in a somewhat less general form.

\subsection{Orthogonal  sums and a ``diagonalisation trick''}\label{sec.b4}
Here we discuss the behaviour of $\Xi$ with respect to orthogonal sum decompositions of the Hilbert space $\HH$. 
First we state a trivial yet useful observation. 
Let $M$ be a lower semibounded self-adjoint operator in $\HH$ and let $A$ be 
a self-adjoint operator which is form-compact with respect to $M$.
Next, let $P$ and $Q$ be orthogonal
projections in $\HH$ such that $P+Q=I$. Assume that $M$ is reduced
by the orthogonal decomposition $\HH=\Ran P\oplus \Ran Q$; 
this means that 
$$
P\Dom(M)\subset\Dom(M), \quad Q\Dom(M)\subset\Dom(M)
$$
and 
$$
PMP\psi=MP\psi, \quad QMQ\psi=MQ\psi,\quad  
\text{ for all $\psi\in\Dom(M)$.}
$$
Then it is easy to see that the operators $PAP$ and $QAQ$ are also 
form compact with respect to $M$ and therefore the form sums $M+PAP$, $M+QAQ$, $M+PAP+QAQ$ 
are well defined. Moreover, one has
\begin{equation}
\Xi(\lambda; M+PAP+QAQ,M)
=
\Xi(\lambda; M+PAP,M)+\Xi(\lambda; M+QAQ,M)
\label{b10}
\end{equation}
for all $\lambda\in\R\setminus\sigma_{ess}(M)$.
This follows directly from the fact that both $M$ and $PAP+QAQ$ are 
reduced by the orthogonal decomposition $\HH=\Ran P\oplus \Ran Q$.

Next, we apply a trick from \cite[Lemma~1.1]{IwaTam}  
to the analysis of $\Xi$. 
This trick is not specific to the function $\Xi$ but rather is a general variational 
consideration. The usefulness of this trick is illustrated by the 
construction of Section~\ref{sec.e}. 
\begin{theorem}\label{th.b4}
Let $M$ be a lower semibounded self-adjoint operator in $\HH$ and let $A$ be 
a self-adjoint operator which is form-compact with 
respect to $M$. Let $P$ and $Q$ be orthogonal
projections in $\HH$ such that $P+Q=I$. 
Assume that $M$ is reduced
by the orthogonal decomposition $\HH=\Ran P\oplus \Ran Q$.
Then for any $\lambda\in\R\setminus\sigma_{ess}(M)$ and for any $\e>0$, one has
\begin{align}
\Xi(\lambda; M+A,M)
&\leq 
\Xi(\lambda; M+P(A+\e \abs{A})P,M)+\Xi(\lambda; M+Q(A+\frac1\e \abs{A})Q,M),
\label{c12}
\\
\Xi(\lambda; M+A,M)
&\geq 
\Xi(\lambda; M+P(A-\e \abs{A})P,M)+\Xi(\lambda; M+Q(A-\frac1\e \abs{A})Q,M).
\label{c13}
\end{align}
\end{theorem}
\begin{proof}
For $\phi,\psi\in\Dom(\abs{A}^{1/2})$, denote
$$
a[\psi,\phi]=(\sign(A)\abs{A}^{1/2}\psi,\abs{A}^{1/2}\phi),
\quad 
\abs{a}[\psi,\phi]=(\abs{A}^{1/2}\psi,\abs{A}^{1/2}\phi).
$$
Let $\psi\in\Dom(\abs{M}^{1/2})$; by assumption, we have
$\psi\in\Dom(\abs{A}^{1/2})$ and 
$P\psi,Q\psi\in\Dom(\abs{M}^{1/2})\subset\Dom(\abs{A}^{1/2})$,
and 
\begin{multline*}
2\aabs{ a[P\psi,Q\psi]}
=
2\aabs{ (\sign(A)\abs{A}^{1/2}P\psi,\abs{A}^{1/2}Q\psi)}
\\
\leq
2\norm{\abs{A}^{1/2}P\psi}\norm{\abs{A}^{1/2}Q\psi}
\leq
\e\abs{a}[P\psi,P\psi]+\frac1\e\abs{a}[Q\psi,Q\psi].
\end{multline*}
It follows that 
\begin{multline*}
a[\psi,\psi]
=
a[P\psi,P\psi]+a[Q\psi,Q\psi]+2\Re a[P\psi,Q\psi]
\\
\leq 
a[P\psi,P\psi]+a[Q\psi,Q\psi]
+\e\abs{a}[P\psi,P\psi]+\frac1\e\abs{a}[Q\psi,Q\psi],
\end{multline*}
and therefore
$$
M+A
\leq 
M+P(A+\e\abs{A})P+Q(A+\frac1\e \abs{A})Q
$$
in the quadratic form sense.
Denote $K=P(A+\e\abs{A})P+Q(A+\frac1\e \abs{A})Q$; using 
Corollary~\ref{cr.b2} and the identity \eqref{b10}, we get
$$
\Xi(\lambda; M+A,M)
\leq 
\Xi(\lambda; M+K,M)
=
\Xi(\lambda; M+PKP,M)+\Xi(\lambda; M+QKQ,M),
$$
which yields  \eqref{c12}. The estimate \eqref{c13} is obtained in a similar way. 
\end{proof}

\section{Generalised Birman-Schwinger principle}\label{sec.g}

\subsection{Statement and discussion}
Let $M$ be a lower semibounded self-adjoint operator in $\HH$, let $A$ be 
a self-adjoint operator which is form-compact with respect to $M$, and let $\wt M$ be defined
as a form sum $\wt M=M+A$. Suppose that $A$ is represented
as $A=G^*JG$, where $G$ is a closed operator from $\HH$ to an auxiliary Hilbert 
space $\K$ such that  for some $\gamma>0$,  $\Dom (M+\gamma I)^{1/2}\subset \Dom G$ and 
$G(M+\gamma I)^{-1/2}$ is compact, and 
$J$ is self-adjoint, bounded in $\K$ and has a bounded inverse. 
(The simplest case of such a factorisation is $\K=\HH$, $A=\abs{A}^{1/2} \sign(A) \abs{A}^{1/2}$.) 
For $\lambda\in\R\setminus\sigma(M)$, define the compact self-adjoint operator 
$T(\lambda)$ in $\K$ by setting
$$
T(\lambda)f=G(M-\lambda I )^{-1}G^*f, \qquad f\in\Dom G^*,
$$
and taking closures.
The following result is essentially due to \cite[Theorem~5.5]{GM}, 
but it has many precursors in the literature, see the discussion below.
\begin{theorem} \label{thm.bs}
Under the above assumptions, for any $\lambda\in\R\setminus\sigma(M)$ one has 
\begin{equation}
\Xi(\lambda; \wt M,M)=\Xi(0; -J^{-1}-T(\lambda), -J^{-1}).
\label{b5}
\end{equation}
\end{theorem}
\begin{remark*}
\begin{enumerate}
\item
The importance of this formula in this context is that the variational estimates for $\Xi$  
from Section~\ref{sec.b}
can now be applied to the 
r.h.s. of \eqref{b5}.
\item
If $(J^{-1}+T(\lambda))$ is invertible, then the identity \eqref{b5} can be rewritten as
\begin{equation}
\Xi(\lambda; \wt M, M)=-\Xi(0; J^{-1}+T(\lambda), J^{-1}).
\label{b28}
\end{equation}
\item
Suppose $\lambda<\inf\sigma(M)$. Then the identity \eqref{b5} 
can be rewritten as the usual Birman-Schwinger principle:
\begin{equation}
N((-\infty,\lambda); \wt M)
=
N((-\infty,-1); X^*JX), 
\quad X=G(M-\lambda I)^{-1/2}\in\frakS_\infty,
\label{b22}
\end{equation}
see \cite{Birman2,Schwinger}.
Indeed, by \eqref{b21} the l.h.s. of \eqref{b22} coincides with minus the l.h.s. of \eqref{b5}.
In order to see that the r.h.s. of \eqref{b22} coincides with minus the r.h.s. of \eqref{b5}, 
let us apply the identity \eqref{b5} to the r.h.s. of itself with $\lambda=0$, 
$M=-J^{-1}$, $A=-T(\lambda)$, $J=-I$, and $G=X^*$:
\begin{multline*}
\Xi(0; (-J^{-1})+X(-I)X^*,(-J^{-1}))
=
\Xi(0; I-X^*(-J^{-1})^{-1}X,I)
\\
=
-N((-\infty,-1); X^*JX).
\end{multline*}
\item
If $J=I$ or $J=-I$, then \eqref{b5} can be rewritten in the following simpler form: 
\begin{align}
\Xi(\lambda; M+G^*G,M)&=N((-\infty,-1]; T(\lambda)),
\label{b23}
\\
\Xi(\lambda; M-G^*G,M)&=-N((1,\infty); T(\lambda)).
\label{b24}
\end{align}
These identities are essentially due to \cite[Theorem~3.5]{Sobolev},
where they were stated in the framework of the spectral shift function theory 
(cf. \eqref{b17a}).
In particular, if $\lambda<\inf\sigma(M)$, then \eqref{b24} becomes
\begin{equation}
N((-\infty,\lambda); M-G^*G)=N((1,\infty); T(\lambda)).
\label{b24a}
\end{equation}
This is perhaps the simplest and the best known case of the Birman-Schwinger
principle. 
\item
If $G$ is a Hilbert-Schmidt class operator (and so $A$ is trace class), 
then \eqref{b5} reduces to a representation for the spectral shift function from \cite{GM}. 
\item
For a discussion of the Birman-Schwinger principle in the context 
of the operator algebras, see \cite{KMS} and references therein.
\end{enumerate}
\end{remark*}

It is not difficult to prove Theorem~\ref{thm.bs} by using the above 
mentioned result from \cite{GM}  and an approximation argument. 
However, \cite{GM} uses some very non-trivial constructions 
from the spectral shift function theory. For this reason, below we give 
an alternative, perhaps more direct proof.

The following corollary is not used in this paper but might be useful 
elsewhere.
\begin{corollary}
Suppose that under the assumptions of Theorem~\ref{thm.bs}, one has
$[\lambda_1,\lambda_2]\subset\R\setminus\sigma_{ess}(M)$
and $\lambda_1,\lambda_2\in\R\setminus\sigma(M)$. Then 
$$
N([\lambda_1,\lambda_2);\wt M)
-
N([\lambda_1,\lambda_2);M)
=
\Xi(0; -J^{-1}-T(\lambda_1),-J^{-1}-T(\lambda_2)).
$$
\end{corollary}
\begin{proof}
By \eqref{b20},  Theorem~\ref{thm.bs}, and the ``chain rule''
\eqref{a8}, one has
\begin{multline*}
N([\lambda_1,\lambda_2);\wt M)
-
N([\lambda_1,\lambda_2);M)
=
\Xi(\lambda_1; \wt M,M)-\Xi(\lambda_2; \wt M,M)
\\
=
\Xi(0; -J^{-1}-T(\lambda_1),-J^{-1})
-
\Xi(0; -J^{-1}-T(\lambda_2),-J^{-1})
\\
=
\Xi(0; -J^{-1}-T(\lambda_1),-J^{-1}-T(\lambda_2)),
\end{multline*}
as required.
\end{proof}

\subsection{Two lemmas}
Here we prove two lemmas which are used in our proof of 
Theorem~\ref{thm.bs}; they might also be of an independent
interest. 
\begin{lemma}\label{lma.g1}
Let $M=M^*$ be a bounded operator which has a bounded inverse. 
Let $X$ be a bounded operator which has a bounded inverse and 
suppose that $X-I$ is compact. Then 
$\Xi(0; XMX^*,M)$ exists and equals zero.
\end{lemma}
\begin{proof}
By assumptions, $XMX^*-M$ is compact; it follows that $\Xi(0; XMX^*,M)$ exists.

Next, 
since the operator $X-I$ is compact, 
one can find a continuous function $f:[0,1]\to\C$, $f(0)=0$, $f(1)=1$, 
such that $X_t=I+f(t)(X-I)$ is invertible for all $t\in[0,1]$. 
Then the family $X_t$, $t\in[0,1]$ is operator norm continuous and 
satisfies conditions $X_0=I$, $X_1=X$, $X_t-I$ is compact 
and $X_t$ is invertible for all $t$.

Let $M_t=X_tMX_t^*$. Then $M_t$ depends continuously on $t$ in the 
operator norm, $M_t$ is invertible for all $t$, and $M_t-M$ is compact 
for all $t$. By the stability Theorem~\ref{stab}, it follows that 
$\Xi(0;M_t,M)$ exists for all $t$, depends continuously on $t$ and 
therefore is constant. 
Finally, for $t=0$ we have $\Xi(0; M_0,M)=\Xi(0;M,M)=0$. 
\end{proof}

The following well known statement provides some insight into the identity \eqref{b5}.
\begin{lemma}\label{lma.b1}
Under the assumptions of Theorem~\ref{thm.bs}, one has
\begin{equation}
\dim\Ker (\wt M-\lambda I)
=
\dim\Ker(J^{-1}+T(\lambda)).
\label{b25a}
\end{equation}
\end{lemma}
\begin{proof}
Without the loss of generality, assume $\lambda=0$. 
Denote $M_0=\sign(M)$ and $Y=G\abs{M}^{-1/2}\in\frakS_\infty$. 
First note that since $\Ker M=\{0\}$, we have
\begin{equation}
\dim\Ker\wt M
=
\dim\Ker\abs{M}^{-1/2}\wt M \abs{M}^{-1/2}
=
\dim\Ker(M_0+Y^*JY).
\label{b26}
\end{equation}

Next, recall the well known fact that for any two compact operators
$K_1$, $K_2$ one has 
$\dim\Ker(K_1 K_2+I)=\dim\Ker(K_2 K_1+I)$.
It follows that 
\begin{equation}
\dim\Ker(Y^*JYM_0+I)=\dim\Ker(YM_0Y^*J+I).
\label{b27}
\end{equation}
The r.h.s. of \eqref{b27} equals 
$$
\dim\Ker((YM_0Y^*+J^{-1})J)=\dim\Ker(T(0)+J^{-1}).
$$
In the same way, the l.h.s. of \eqref{b27} equals 
$\dim\Ker(M_0+Y^*J Y)$. 
Together with \eqref{b26}, this proves the claim. 
\end{proof}

\subsection{The proof of Theorem~\ref{thm.bs}}
1. 
Without the loss of generality, let us assume $\lambda=0$.
First let us prove \eqref{b5} under the additional assumptions that 
$M$ is bounded and  $0\notin\sigma(\wt M)$.
If $M$ is bounded then our original assumption $G(M+\gamma I)^{-1/2}\in\frakS_\infty$ 
means simply that $G$ is compact. 

Consider the following bounded operator 
in $\HH\oplus\K$:
$$
X=
\begin{pmatrix}
I & -G^*J
\\
G M^{-1} & I
\end{pmatrix}.
$$
It is straightforward to see that $X-I$ is compact.
One can directly verify the identity
\begin{equation}
X
\begin{pmatrix}
M & 0 \\
0 & J^{-1} 
\end{pmatrix}
X^*
=
\begin{pmatrix}
\wt M  & 0 
\\
0 & J^{-1} + T(0)
\end{pmatrix}.
\label{b35}
\end{equation}
By our assumption $0\notin\sigma(\wt M)$ and Lemma~\ref{lma.b1}, 
the operator on the r.h.s. has a bounded inverse. 
From here and the compactness of $X-I$ it follows that 
$X$ has a bounded inverse. 

Next, 
from \eqref{b35}  and Lemma~\ref{lma.g1} it follows that 
$$
\Xi(0; \wt M,M)+
\Xi(0; J^{-1}+T(0),J^{-1})=0.
$$
This is the same as \eqref{b28}.

2. While still assuming that $0\notin\sigma(\wt M)$, 
let us lift the assumption of boundedness of $M$. 
Let $M$ be unbounded, and let $P_n=E_M(-n,n)$, 
$G_n=GP_n\in\frakS_\infty$, $A_n=G_n^*JG_n$, $T_n(z)=G_n(M-zI)^{-1}G_n^*$. 
In what follows, we will prove that 
\begin{gather}
\norm{T_n(z)-T(z)}\to 0\quad\text{as $n\to\infty$, $\forall z\notin\sigma(M)$,}
\label{b36}
\\
M+A_n\to\wt M\quad\text{as $n\to\infty$ in the norm resolvent sense.}
\label{b37}
\end{gather}
Since $\wt M$ and (by Lemma~\ref{lma.b1}) $J^{-1}+T(0)$ are invertible, 
from \eqref{b36}, \eqref{b37} it follows, in particular,  that $M+A_n$ 
and $J^{-1}+T_n(0)$ are invertible for all sufficiently large $n$. 

Next, note that the orthogonal decomposition $\HH=\Ran P_n\oplus\Ran(I-P_n)$ 
reduces both $M$ and $M+A_n$. The components of both $M$ and $M+A_n$ 
in $\Ran P_n$ are bounded. The components of $M$ and $M+A_n$ in 
$\Ran (I-P_n)$ coincide. 
By the first step of the proof, it follows that 
\begin{equation}
\Xi(0; M+A_n,M)=-\Xi(0; J^{-1}+T_n(0),J^{-1})
\label{b38}
\end{equation}
for all sufficiently large $n$. 
Now by \eqref{b36} and \eqref{b37} and the stability Theorem~\ref{stab}, 
we can pass to the limit as $n\to\infty$ in \eqref{b38}, which yields \eqref{b28}.
Note that it is easy to see that the uniform lower bound assumption from 
Theorem~\ref{stab} is satisfied in our case. 

3. Let us prove the convergence \eqref{b36} and \eqref{b37}.
First note that since 
$G(M+\gamma I)^{-1/2}$ is compact and $P_n\to I$ strongly as $n\to\infty$, 
we obtain
\begin{gather}
\norm{G_n(M+\gamma I)^{-1/2}-G(M+\gamma I)^{-1/2}}\to 0,
\quad n\to\infty,
\label{b39}
\\
\norm{G_n(M\pm i)^{-1}-G(M\pm i)^{-1}}\to 0,
\quad n\to\infty.
\label{b40}
\end{gather}
Writing 
\begin{gather*}
T(z)=[G(M+\gamma I)^{-1/2}](M+\gamma I)(M-zI)^{-1}[G(M+\gamma I)^{-1/2}]^*,
\\
T_n(z)=[G_n(M+\gamma I)^{-1/2}](M+\gamma I)(M-zI)^{-1} [G_n(M+\gamma I)^{-1/2}]^*
\end{gather*} 
and using \eqref{b39}, we obtain \eqref{b36}. 
In order to prove the convergence \eqref{b37}, we use the 
iterated resolvent identities in the form
\begin{equation}
(\wt M-i)^{-1}=(M-i)^{-1}-[G(M+i)^{-1}]^*(J^{-1}+T(i))^{-1}[G(M-i)^{-1}],
\label{g6}
\end{equation}
and similarly
\begin{equation}
(M+A_n-i)^{-1}=(M-i)^{-1}-[G_n(M+i)^{-1}]^*(J^{-1}+T_n(i))^{-1}[G_n(M-i)^{-1}].
\label{g7}
\end{equation}
Subtracting and using \eqref{b36} and \eqref{b40}, we obtain \eqref{b37}.

4. 
It remains to lift the assumption $0\notin\sigma(\wt M)$.
Suppose $0\in\R\setminus\sigma(M)$ and 
$0\in\sigma(\wt M)$; then $0$ is an isolated
eigenvalue of $\wt M$. 
It suffices to prove that both sides of \eqref{b5} are left continuous 
in $\lambda$ at $\lambda=0$.
For the l.h.s. of \eqref{b5}, this is true directly by the definition of $\Xi$. 
Let us consider the r.h.s. 

We claim that there exist $\varepsilon>0$ and $\delta>0$ such that 
\begin{equation}
\Xi(0; -J^{-1}-T(\lambda),-J^{-1})=
\Xi(-\delta; -J^{-1}-T(\lambda),-J^{-1}),
\qquad \forall \lambda\in[-\varepsilon,0].
\label{g8}
\end{equation}

Indeed, it is easy to see that $T(\lambda)$ is continuous in $\lambda$ 
in the operator norm at $\lambda=0$ and $T(\lambda)\leq T(0)$
for small $\lambda\leq0$.
Since $T(\lambda)$ is compact, it follows that for some $\delta>0$ 
and all sufficiently small $\lambda\leq 0$, one has 
$$
[-\delta,0)\cap\sigma(-J^{-1}-T(\lambda))=\varnothing
\quad
\text{ and }
\quad
[-\delta,0]\cap \sigma(-J^{-1})=\varnothing.
$$
Then \eqref{g8} follows.

Now by the stability Theorem~\ref{stab}, we have
$$
\Xi(-\delta; -J^{-1}-T(\lambda),-J^{-1})
\to
\Xi(-\delta; -J^{-1}-T(0),-J^{-1})
$$
as $\lambda\to0$. 
It follows that the r.h.s. of \eqref{b5}  is left continuous in 
$\lambda$ at $\lambda=0$, as required. 
\qed

\section{Safronov's theorem}\label{sec.d}
\subsection{The key estimates}
Here we state and prove a result (see \eqref{d11}, \eqref{d12} below) which is a slight generalisation of \cite{Safronov2}
(see also \cite{Safronov1}). 
Let $H_0$ be a lower semi-bounded self-adjoint operator in $\HH$; choose $\gamma\in\R$ such that 
$H_0+\gamma I \geq I$. Let $V_+\geq0$ and $V_-\geq0$  be self-adjoint operators in $\HH$ 
which are form-compact with respect to $H_0$.
Then for any $t\in\R$ the operators $H_0+tV_+-tV_-$ are well defined in terms of the 
corresponding quadratic forms.
Fix $\lambda\in\R\setminus\sigma(H_0)$; our aim is to consider
\begin{equation}
\Xi(\lambda; H_0+tV_+-tV_-,H_0)
\quad \text{ for $t\to\infty$.}
\label{d0}
\end{equation}
Let us define the auxiliary compact operators $T_{\alpha\beta}$,  $\alpha, \beta\in\{+,-\}$, by setting
$$
T_{\alpha \beta}f=\sqrt{V_\alpha}(H_0-\lambda I)^{-1}\sqrt{V_\beta}f, 
\quad f\in \Dom(H_0+\gamma I)^{1/2}
$$
and taking closures. Clearly, $T_{++}$ and $T_{--}$ are self-adjoint and $T_{+-}^*=T_{-+}$.
We note that, by a well known identity, 
\begin{equation}
N((a, \infty); T_{-+}T_{+-})=N((a, \infty); T_{+-}T_{-+})
\label{d0a}
\end{equation}
for any $a>0$. 
The key fact is the following 
\begin{theorem}\label{lma.d1}
For any $a\in (0,1)$, one has 
\begin{multline}
\Xi(\lambda; H_0+V_+-V_-, H_0)
\leq
\Xi(\lambda; H_0+\frac{1}{1-a}V_+,H_0)
\\
+\Xi(\lambda; H_0-\frac{1}{1+a}V_-,H_0)
+
N((a^2, \infty); T_{-+}T_{+-}),
\label{d9}
\end{multline}
\begin{multline}
\Xi(\lambda; H_0+V_+-V_-, H_0)
\geq
\Xi(\lambda; H_0+\frac{1}{1+a}V_+,H_0)
\\
+\Xi(\lambda; H_0-\frac{1}{1-a}V_-,H_0)
-
N((a^2, \infty); T_{-+}T_{+-}).
\label{d10}
\end{multline}
\end{theorem}
\begin{proof}
1. Writing
$$
\frac{1}{1-a}V_+=\sqrt{V_+}(1-a)^{-1}\sqrt{V_+},
\qquad
-\frac{1}{1+a}V_-=\sqrt{V_-}(-1-a)^{-1}\sqrt{V_-},
$$
and using the generalised Birman-Schwinger principle \eqref{b5}, we get
\begin{align}
\Xi(\lambda; H_0+\frac{1}{1-a}V_+, H_0)
&=
\Xi(0; -I+aI-T_{++},-I+aI)
=\Xi(-a;-I-T_{++},-I),
\label{d3} 
\\
\Xi(\lambda; H_0-\frac{1}{1+a}V_-, H_0)
&=
\Xi(0; I+aI-T_{--},I+aI)
=\Xi(-a;I-T_{--},I).
\label{d4} 
\end{align}

2. Let $\K=\HH\oplus\HH$ and let $G:\HH\to\K$ be the closure of the operator defined by 
$$
Gf=(\sqrt{V_+}f,\sqrt{V_-}f),
\quad 
f\in\Dom(H_0+\gamma I)^{1/2}.
$$
Let $J:\K\to\K$ be the operator $(f_+,f_-)\mapsto(f_+,-f_-)$.
Then one has 
$$
\norm{\sqrt{V_+}f}_{\HH}^2-\norm{\sqrt{V_-}f}_{\HH}^2
=
(JGf,Gf)_{\K}, 
\quad f\in\Dom(H_0+\gamma I)^{1/2}, 
$$
and so 
$H_0+V_+-V_-=H_0+G^*JG$. 
Let us define a compact operator $\T$ in $\K$ by setting 
$$
\T f=G(H_0-\lambda I)^{-1}G^*f, \quad f\in\Dom(H_0+\gamma I)^{1/2}, 
$$
and taking closures.
An application of the generalised Birman-Schwinger principle  \eqref{b5}
yields
\begin{equation}
\Xi(\lambda; H_0+V_+-V_-, H_0)
=
\Xi(0; -J-\T , -J);
\label{d5a}
\end{equation}
note that here $J^{-1}=J$. 

3. 
Let $P_\pm:\K\to\K$ be the orthogonal projections, 
$$
P_+: (f_+,f_-)\mapsto(f_+,0),
\qquad
P_-: (f_+,f_-)\mapsto(0,f_-),
$$
and let 
$\T_{\alpha \beta}=P_\alpha \T P_\beta$, 
$\alpha, \beta\in\{-,+\}$.
Applying Corollary~\ref{cr.b3} followed by \eqref{b10}
and then using \eqref{d3}, \eqref{d4}, 
we get
\begin{multline}
\Xi(0; -J-\T, -J)
\leq 
\Xi(-a; -J-\T_{++}-\T_{--}, -J)
+N((a,\infty); -\T_{+-}-\T_{-+})
\\
=
\Xi(-a; -I-T_{++},-I)+\Xi(-a; I-T_{--},I)
+N((a,\infty); -\T_{+-}-\T_{-+})
\\
=
\Xi(\lambda; H_0+\frac{1}{1-a}V_+,H_0)
+\Xi(\lambda; H_0-\frac{1}{1+a}V_-,H_0)
+N((a,\infty); -\T_{+-}-\T_{-+}).
\label{d6}
\end{multline}

4. 
By a direct calculation,
$\T_{+-}+\T_{-+}=-J(\T_{+-}+\T_{-+})J$
and therefore
$$
N((a,\infty); \T_{+-}+\T_{-+})=N((a,\infty); -\T_{+-}-\T_{-+}).
$$
Using this fact and \eqref{d0a}, we get
\begin{multline}
2N((a,\infty);  \T_{+-}+\T_{-+})
=
N((a^2,\infty); (\T_{+-}+\T_{-+})^2)
\\
=
N((a^2,\infty); T_{-+}T_{+-})
+
N((a^2,\infty); T_{+-}T_{-+})
=
2N((a^2,\infty); T_{-+}T_{+-}).
\label{d6a}
\end{multline}
Combining \eqref{d5a}, \eqref{d6}, \eqref{d6a}, 
we obtain the upper bound \eqref{d9}.
The lower bound \eqref{d10} is proven in the same manner. 
\end{proof}

\subsection{Applications}
From Theorem~\ref{lma.d1} one easily obtains the main result of \cite{Safronov2}.
In  \cite{Safronov2}, the asymptotics \eqref{d0} was studied under the 
assumption that 
$$
\limsup_{t\to\infty}t^{-p}
N((1, \infty); t^2T_{-+}T_{+-})=0
$$
for some exponent $p>0$. 
Combining Theorem~\ref{lma.d1} with this assumption,  we obtain
\begin{multline*}
\limsup_{t\to\infty} t^{-p}\Xi(\lambda; H_0+tV_+-tV_-,H_0)
\leq 
\limsup_{t\to\infty} t^{-p}\Xi(\lambda; H_0+\frac{t}{1-a}V_+,H_0)
\\
+
\limsup_{t\to\infty} t^{-p}\Xi(\lambda; H_0-\frac{t}{1+a}V_-,H_0)
+
\limsup_{t\to\infty}t^{-p}N((1,\infty); \frac{t^2}{a^2}T_{-+}T_{+-})
\\
=
(1-a)^{-p}\limsup_{t\to\infty} t^{-p}\Xi(\lambda; H_0+tV_+,H_0)
+
(1+a)^{-p}\limsup_{t\to\infty} t^{-p}\Xi(\lambda; H_0-tV_-,H_0).
\end{multline*}
Letting $a\to0$, we obtain
\begin{multline}
\limsup_{t\to\infty} t^{-p}\Xi(\lambda; H_0+tV_+-tV_-,H_0)
\leq 
\limsup_{t\to\infty} t^{-p}\Xi(\lambda; H_0+tV_+,H_0)
\\
+
\limsup_{t\to\infty} t^{-p}\Xi(\lambda; H_0-tV_-,H_0),
\label{d11}
\end{multline}
and in the same way 
\begin{multline}
\liminf_{t\to\infty} t^{-p}\Xi(\lambda; H_0+tV_+-tV_-,H_0)
\geq 
\liminf_{t\to\infty} t^{-p}\Xi(\lambda; H_0+tV_+,H_0)
\\
+
\liminf_{t\to\infty} t^{-p}\Xi(\lambda; H_0-tV_-,H_0).
\label{d12}
\end{multline}
The estimates \eqref{d11}, \eqref{d12} were obtained in 
\cite{Safronov2} (see also \cite{Safronov1}) by a different method.
These estimates were then applied in  \cite{Safronov1,Safronov2} 
to various cases when $H_0$ is a differential operator and $V$ 
is the operator of multiplication by a function from an appropriate $L^q$ class.
Let us quote a typical application:
$$
\lim_{t\to\infty} t^{-d/2} \Xi(\lambda; -\Delta+V_0+tV,-\Delta+V_0)
=
-(2\pi)^{-d}\omega_d\int_{\{x: V(x)\leq 0\}} \abs{V(x)}^{d/2} dx,
$$
where $d\geq3$, $V\in L^{d/2}(\R^d)$, $V_0$ is a bounded potential in $\R^d$, 
$\lambda$ is in a gap of the spectrum of $-\Delta+V_0$ and $\omega_d$ 
is the volume of a unit ball in $\R^d$. 
Here the first terms in the r.h.s. of \eqref{d11}, \eqref{d12} vanish and the 
second terms coincide and are independent of $\lambda$. 
The analysis of the second terms uses the Birman-Schwinger principle 
in the form \eqref{b24} and the technique of \cite{Birman}.

We note that in \cite{Safronov1,Safronov2} the results are stated in terms of 
the spectral flow of the family $H_0+tV_+-tV_-$, see \eqref{a1}.
We also note that the arguments similar to the ones of the proof of 
Theorem~\ref{lma.d1} were used in \cite{Safronov3} in the analysis of the 
asymptotics of the spectral shift function.

\section{An abstract theorem}\label{sec.e}

Here we state and prove a theorem which will be used in the next section in application
to the study of the eigenvalues of the perturbed Landau Hamiltonian. 
This theorem (or the method of its proof) might also be useful in applications
to the  perturbed periodic operator.

\subsection{The statement of the Theorem}
Let $H_0$ be a lower semi-bounded self-adjoint operator in a Hilbert space $\HH$. 
Let $P_0,P_1,P_2,\dots$ be a sequence of orthogonal projections in $\HH$ such that 
$P_nP_m=0$ for all $n\not=m$ and $\sum_{n=0}^\infty P_n=I$. Assume also that
the orthogonal   decomposition $\HH=\oplus_{n=0}^\infty\Ran P_n$ reduces  
$H_0$  and 
\begin{equation}
\inf\sigma(H_0|_{\Ran P_n})\to\infty
\quad \text{as} \quad n\to\infty.
\label{e1}
\end{equation}
Let $V_t$, $t>0$ be a family of self-adjoint operators such that for all $t>0$, 
$V_t$ is form compact with respect to $H_0$.
Below we consider the asymptotics of 
\begin{equation}
\Xi(\lambda; H_0+V_t,H_0)
\quad \text{ as $t\to\infty$.}
\label{e1a}
\end{equation}
The perturbation $V_t$ can be regarded as a ``matrix''
$\{P_nV_tP_m\}$  with respect to the orthogonal sum decomposition 
$\HH=\oplus_{n=0}^\infty\Ran P_n$.
Under the appropriate assumptions, below we prove that, roughly speaking,
only the diagonal terms of this ``matrix'' contribute to the asymptotics \eqref{e1a}.

For some exponent $p>0$, we assume
\begin{gather}
\limsup_{t\to\infty} t^{-p}N((a,\infty); \pm( P_n\abs{V_t}P_m+P_m\abs{V_t}P_n) )=0
\quad \text{for any $a>0$ and $n\not=m$,}
\label{e2}
\\
\limsup_{t\to\infty} t^{-p}N((a,\infty); \pm( P_n V_t P_m+P_m V_t P_n) )=0
\quad \text{for any $a>0$ and $n\not=m$,}
\label{e2a}
\\
\lim_{E\to\infty}\limsup_{t\to\infty}t^{-p}N((-\infty,-E); H_0-s\abs{V_t})=0
\quad \text{for any $s>0$.}
\label{e3}
\end{gather}
\begin{theorem}\label{th.e1}
Let the above assumptions hold true. 
Then for any $\lambda\in\R\setminus\sigma(H_0)$,
any $\e>0$ and any  $a>0$ sufficiently small such that 
$[\lambda-a,\lambda+a]\cap \sigma(H_0)=\varnothing$, 
one has
\begin{align}
\limsup_{t\to\infty}t^{-p}\Xi(\lambda; H_0+V_t,H_0)
\leq
\sum_{n=0}^\infty
\limsup_{t\to\infty} t^{-p}\Xi(\lambda-a; H_0+P_n(V_t+\e \abs{V_t})P_n,H_0),
\label{e4}
\\
\liminf_{t\to\infty}t^{-p}\Xi(\lambda; H_0+V_t,H_0)
\geq
\sum_{n=0}^\infty
\liminf_{t\to\infty} t^{-p}\Xi(\lambda+a; H_0+P_n(V_t-\e \abs{V_t})P_n,H_0).
\label{e5}
\end{align}
\end{theorem}
\begin{remark*}
\begin{enumerate}
\item
If $n$ is sufficiently large so that $\lambda-a<\inf\sigma(H_0 |_{\Ran P_n})$, 
then, using the orthogonal decomposition
$\HH=\Ran P_n\oplus \Ran(I-P_n)$ and \eqref{b21}, we easily obtain
\begin{equation}
\Xi(\lambda-a; H_0+P_n(V_t\pm\varepsilon \abs{V_t})P_n,H_0)
\leq0.
\label{e5a}
\end{equation}
Thus, all terms in the series \eqref{e4}, \eqref{e5} 
with sufficiently large $n$ are non-positive. 
\item
Assumption \eqref{e3} is used only in the proof of \eqref{e5}.
\end{enumerate}
\end{remark*}

\subsection{Proof of Theorem~\ref{th.e1}}
For $r\geq1$, denote $P^{(r)}=\sum_{n=0}^rP_n$, $Q^{(r)}=I-P^{(r)}$. 

\emph{Upper bound:} 
1. 
Using the ``diagonalisation trick" (Theorem~\ref{th.b4}), we obtain
\begin{multline}
\Xi(\lambda; H_0+V_t,H_0)
\\
\leq
\Xi(\lambda; H_0+P^{(r)}(V_t+\e \abs{V_t})P^{(r)},H_0)
+
\Xi(\lambda; H_0+Q^{(r)}(V_t+\frac1\e \abs{V_t})Q^{(r)},H_0)
\label{e8}
\end{multline}
for any $r$.
Next, as in \eqref{e5a}, 
we see that if 
$r$ is sufficiently large, then 
$$
\Xi(\lambda; H_0+Q^{(r)}(V_t+\frac1\e \abs{V_t})Q^{(r)},H_0)\leq 0.
$$
Thus, from \eqref{e8} we obtain the estimate
\begin{equation}
\Xi(\lambda; H_0+V_t,H_0)
\leq 
\Xi(\lambda; H_0+P^{(r)}(V_t+\e \abs{V_t})P^{(r)},H_0).
\label{e9}
\end{equation}

2. In what follows, we use the iterated Weyl's inequality 
(see e.g. \cite[Section~11.1]{BS})
for the eigenvalues of compact selfadjoint operators
$K_1$, $K_2$, \dots, $K_\ell$:
\begin{equation}
N((a,\infty); K_1+\dots+K_\ell)
\leq
\sum_{j=1}^\ell N((a/\ell,\infty); K_j).
\label{weyl}
\end{equation}
Let us  write $W_t=V_t+\e \abs{V_t}$ and 
$$
P^{(r)}W_tP^{(r)}=W_t^{(diag)}+W_t^{(off)},
$$
where 
$$
W_t^{(diag)}=\sum_{n\leq r} P_n W_t P_n,
\quad
W_t^{(off)}=
\sum_{\genfrac{}{}{0pt}{}{n\not=m} {n,m\leq r}}  P_n W_t P_m.
$$
By \eqref{weyl}, we get for any $a>0$: 
\begin{multline*}
N((a,\infty); W_t^{(off)})
\leq
\sum_{n<m\leq r}
N((\frac{a}{r(r-1)},\infty); P_nV_tP_m+P_mV_tP_n)
\\
+
\sum_{n<m\leq r}
N((\frac{a}{\varepsilon r(r-1)},\infty); P_n\abs{V_t}P_m+P_m\abs{V_t}P_n).
\end{multline*}
From here, using \eqref{e2}, \eqref{e2a}, we get 
\begin{equation}
\limsup_{t\to\infty}t^{-p}N((a,\infty); W_t^{(off)})=0.
\label{e9a}
\end{equation}

3. 
Using Corollary~\ref{cr.b3} and \eqref{b10}, for any sufficiently small $a>0$ we obtain
\begin{multline*}
\Xi(\lambda; H_0+P^{(r)}(V_t+\e \abs{V_t})P^{(r)},H_0)
\leq
\Xi(\lambda-a; H_0+W_t^{(diag)},H_0)
+
N((a,\infty),W_t^{(off)})
\\
=
\sum_{n\leq r}\Xi(\lambda-a; H_0+P_n W_t P_n,H_0)
+
N((a,\infty),W_t^{(off)}).
\end{multline*}
Using \eqref{e9a}, this yields
$$
\limsup_{t\to\infty} t^{-p}\Xi(\lambda; H_0+V_t,H_0)
\leq 
\sum_{n\leq r}\limsup_{t\to\infty} t^{-p}
\Xi(\lambda-a;H_0+P_n(V_t+\e \abs{V_t})P_n,H_0).
$$
Since $r$ can be chosen arbitrary large, we get the upper bound \eqref{e4}. 

\emph{Lower bound:}
As in \eqref{e8}, we get 
\begin{equation}
\Xi(\lambda; H_0+V_t,H_0)
\geq
\Xi(\lambda; H_0+P^{(r)}(V_t-\e \abs{V_t})P^{(r)},H_0)
+
\Xi(\lambda; H_0+Q^{(r)}(V_t-\frac1\e \abs{V_t})Q^{(r)},H_0).
\label{e10}
\end{equation}
Consider the two terms in the r.h.s. of \eqref{e10}. 
For the first term, as in the proof of the upper bound, we get
\begin{multline}
\liminf_{t\to\infty} t^{-p}\Xi(\lambda; H_0+P^{(r)}(V_t-\e \abs{V_t})P^{(r)},H_0)
\\
\geq 
\sum_{n\leq r}
\liminf_{t\to\infty} t^{-p}\Xi(\lambda+a; H_0+P_n(V_t-\e \abs{V_t})P_n,H_0).
\label{e11}
\end{multline}
Consider the second term.
Denote 
$H_0^{(r)}=H_0|_{\Ran Q^{(r)}}$, 
$V_t^{(r)}= Q^{(r)}V_t Q^{(r)}|_{\Ran  Q^{(r)}}$, 
$\abs{V_t}^{(r)}= Q^{(r)} \abs{V_t} Q^{(r)}|_{\Ran  Q^{(r)}}$,
$\Lambda_r=\inf\sigma(H_0^{(r)})$. 
By the assumption \eqref{e1}, we have $\Lambda_r\to\infty$ as $r\to\infty$. 
If $r$ is sufficiently large so that $\lambda<\Lambda_r$, then (similarly to \eqref{e5a}),
$$
\Xi(\lambda; H_0+ Q^{(r)}(V_t-\frac1\varepsilon\abs{V_t}) Q^{(r)}, H_0)
=
-N((-\infty,\lambda); H_0^{(r)}+V_t^{(r)}-\frac1\varepsilon\abs{V_t}^{(r)}).
$$
Next, by variational considerations, we have
\begin{multline*}
N((-\infty,\lambda); H_0^{(r)}+V_t^{(r)}-\frac1\varepsilon\abs{V_t}^{(r)})
\leq 
N((-\infty,\lambda); H_0^{(r)}-\frac{1+\varepsilon}{\varepsilon}\abs{V_t}^{(r)})
\\
=
N((-\infty,2\lambda); 2H_0^{(r)}-2\frac{1+\varepsilon}{\varepsilon}\abs{V_t}^{(r)})
\leq
N((-\infty,2\lambda); H_0^{(r)}+\Lambda_r-2\frac{1+\varepsilon}{\varepsilon}\abs{V_t}^{(r)})
\\
\leq
N((-\infty,2\lambda-\Lambda_r); H_0^{(r)}-2\frac{1+\varepsilon}{\varepsilon}\abs{V_t}^{(r)}).
\end{multline*}
From here, by assumption \eqref{e3}, we obtain
\begin{equation}
\lim_{r\to\infty}\limsup_{t\to\infty} t^{-p} 
\abs{\Xi(\lambda; H_0+Q^{(r)}(V_t-\frac1\e \abs{V_t})Q^{(r)},H_0)}=0.
\label{e12}
\end{equation}
Combining \eqref{e10}, \eqref{e11}, and \eqref{e12}, we obtain the lower bound \eqref{e5}.
\qed

\section{A theorem of Rozenblum and Sobolev}\label{sec.c}

\subsection{Introduction}
In $L^2(\R^2,dx_1dx_2)$, consider the Landau operator
$$
H_0=
\biggl(-i\frac{\partial}{\partial x_1}+\frac12 Bx_1\biggr)^2
+
\biggl(-i\frac{\partial}{\partial x_2}-\frac12 Bx_2\biggr)^2,
\quad 
B>0.
$$
It is well known that the spectrum of $H_0$ consists of a sequence of infinitely degenerate
eigenvalues (Landau levels) $\Lambda_n=B(2n+1)$, $n=0,1,\dots$; 
we set $\Lambda_{-1}=-\infty$ for notational convenience. 
We denote by $P_n$ the orthogonal projection onto the eigenspace $\Ker(H_0-\Lambda_n I)$. 

Let $V\in L^1(\R^2)\cap L^2(\R^2)$ be a real valued function. 
It is easy to see that the operator of multiplication by $V$ is form-compact with respect to $H_0$.
For $\alpha>0$ and $\beta>0$, we consider the spectral asymptotics of the operator
$$
H_t=H_0+V_t, 
\quad 
V_t(x_1,x_2)=V(x_1 t^{-\alpha},x_2 t^{-\beta})
$$
as $t\to\infty$. We use the notation $p=\alpha+\beta$ and 
$$
A(a,V)=\frac{B}{2\pi}\meas\{ x\in\R^2\mid V(x)>a\}, 
\quad
A[a,V]=\frac{B}{2\pi}\meas\{ x\in\R^2\mid V(x)\geq a\},
$$
where $a>0$ and $\meas$ is the Lebesgue measure in $\R^2$.
Our main result in this section is 
\begin{theorem}\label{th1}
Under the above assumptions, for any  $q\geq-1$ and 
any  $\lambda\in(\Lambda_q,\Lambda_{q+1})$, 
one has
\begin{align}
\limsup_{t\to\infty} t^{-p}\Xi(\lambda; H_t,H_0)
&\leq
\sum_{0\leq n\leq q} A[\lambda-\Lambda_n,V]
-
\sum_{n\geq q+1} A(\Lambda_n-\lambda,-V),
\label{a2}
\\
\liminf_{t\to\infty} t^{-p}\Xi(\lambda; H_t,H_0)
&\geq
\sum_{0\leq n\leq q} A(\lambda-\Lambda_n,V)
-
\sum_{n\geq q+1} A[\Lambda_n-\lambda,-V].
\label{a3}
\end{align}
Further, for any interval 
$[\lambda_1,\lambda_2]\subset \R\setminus\sigma(H_0)$,
one has 
\begin{align}
\limsup_{t\to\infty} t^{-p}N([\lambda_1,\lambda_2]; H_t)
&\leq
\frac{B}{2\pi}\sum_{n\geq 0} \meas\{x\in\R^2\mid \lambda_1-\Lambda_n\leq V(x)\leq  \lambda_2-\Lambda_n\},
\label{a5}
\\
\liminf_{t\to\infty} t^{-p}N((\lambda_1,\lambda_2); H_t)
&\geq
\frac{B}{2\pi}\sum_{n\geq 0} \meas\{x\in\R^2\mid \lambda_1-\Lambda_n< V(x)<  \lambda_2-\Lambda_n\}.
\label{a6}
\end{align}
\end{theorem}

In \cite{RS}, the asymptotic estimates \eqref{a5}, \eqref{a6} 
were proven for $\alpha=\beta=1$. Our construction is somewhat
more direct than the one of \cite{RS}. 
The operator theoretic component of our construction is Theorem~\ref{th.e1}. 
The other component is the analysis of the spectral asymptotics 
of the operators $P_nV_tP_m$ in Section~\ref{sec.c2}. 

\subsection{Spectral asymptotics of the operators $P_nV_tP_m$} \label{sec.c2}
Here we prove 
\begin{proposition}\label{prp.c2}
For any $V\in L^1(\R^2)\cap L^2(\R^2)$ and any $a>0$, one has
\begin{align}
&\limsup_{t\to\infty} t^{-p}N([a,\infty); P_n V_tP_n)\leq A[a,V],
\quad \forall n\geq0,
\label{c1}
\\
&\liminf_{t\to\infty} t^{-p} N((a,\infty); P_n V_t P_n)\geq A(a,V),
\quad \forall n\geq0,
\label{c2}
\\
&\limsup_{t\to\infty} t^{-p}N((a,\infty); P_m V_t P_n+P_n V_tP_m)=0,
\quad \forall n\not = m.
\label{c3}
\end{align}
\end{proposition}
The proof of this Proposition follows an unpublished remark by 
A.~Laptev and Yu.~Safarov and their earlier work \cite{LS}.
This Proposition was also proved in \cite{RS} by a different method.

The analysis below uses the well known explicit formula for the integral kernel of $P_n$:
\begin{equation}
P_n(x,y)=\frac{B}{2\pi} L_n(\tfrac{B\abs{x-y}^2}{2}) \exp(-\frac{B}4(\abs{x-y}^2+2i[x,y]),
\label{f1}
\end{equation}
where $L_n$ is the Laguerre polynomial and $[x,y]=x_1y_2-x_2y_1$. 
\begin{lemma}\label{lma.f1}
For any $V\in L^1(\R^2)\cap L^2(\R^2)$ and any $n=0,1,2,\dots$, one has 
$$
\norm{P_n V_t-V_t P_n}_{\frakS_2}^2=o(t^p), 
\quad t\to\infty.
$$
\end{lemma}
\begin{proof}
The integral kernel of $P_n V_t-V_t P_n$ is 
$P_n(x,y)(V_t(y)-V_t(x))$.
Using this fact, formula \eqref{f1}, and the obvious estimate
$$
\abs{L_n(s)} e^{-s/4}\leq C=C(n), \quad s\geq 0
$$
for the Laguerre polynomial, we obtain
\begin{multline}
\norm{P_n V_t-V_t P_n}_{\frakS_2}^2
\leq 
\left( \frac{B}{2\pi}\right)^2 C^2 \int_{\R^2}\int_{\R^2} dx dy \  \abs{V_t(x)-V_t(y)}^2 
\exp(-\tfrac{B}{4}\abs{x-y}^2)
\\
=
\left( \frac{B}{2\pi}\right)^2 C^2 t^{2p} \int_{\R^2}\int_{\R^2} dx dy \  \abs{V(x)-V(y)}^2 
\exp\bigl(-\tfrac{B}{4} t^{2\alpha}(x_1-y_1)^2-\tfrac{B}{4} t^{2\beta}(x_2-y_2)^2\bigr)
\\
=
t^{2p}\int_{\R^2} dz \  f(z) 
\exp\bigl(-\tfrac{B}{4} t^{2\alpha}z_1^2-\tfrac{B}{4} t^{2\beta}z_2^2\bigr),
\label{f3}
\end{multline}
where
$$
f(z)=\left( \frac{B}{2\pi}\right)^2 C^2
\int dx \  \abs{V(x)-V(x+z)}^2.
$$
It is easy to see that $f(z)$ is continuous in $z$,  $f(0)=0$, and
$$
f(z)
\leq 
4\left( \frac{B}{2\pi}\right)^2 C^2 \int_{\R^2} V(x)^2 dx.
$$
Given $\varepsilon>0$, let us choose $\delta>0$ such that $\abs{f(z)}\leq \varepsilon$
for $\abs{z}\leq \delta$. 
Then, splitting the integral in the r.h.s. of \eqref{f3} into the sum of the integrals 
over $\{z: \abs{z}\leq \delta\}$ and over $\{z: \abs{z}> \delta\}$, 
we readily obtain the estimate
$$
\norm{P_n V_t-V_t P_n}_{\frakS_2}^2
\leq 
C_1\varepsilon t^p + C_2 t^{2p} \int_{\abs{z}>\delta} 
\exp\bigl(-\tfrac{B}{4} t^{2\alpha}z_1^2-\tfrac{B}{4} t^{2\beta}z_2^2\bigr) dz.
$$
Since $\varepsilon>0$ is arbitrary and the integral in the r.h.s. tends to zero 
faster than any power of $t$ as $t\to\infty$, this yields the required estimate. 
\end{proof}
\begin{lemma}\label{lma.f2}
Let $\phi$ be a function from the Sobolev class $W_\infty^2(\R)$ 
(i.e. $\phi''\in L^\infty(\R)$) such that $\phi(0)=0$. 
Then 
$$
\Tr \phi(P_n V_t P_n)
=
\frac{B}{2\pi} t^p \int_{\R^2} \phi(V(x)) dx +o(t^p), 
\quad t\to\infty.
$$
\end{lemma}
\begin{proof}
By \cite[Theorem~1.2]{LS}, we have an estimate  
$$
\aabs{\Tr(P_n\phi(P_nV_tP_n)P_n)-\Tr(P_n\phi(V_t)P_n)}
\leq
\frac12\norm{\phi''}_{L^\infty(\R)}\norm{P_n V_t (I-P_n)}^2_{\frakS_2}.
$$
This estimate has a general operator theoretic nature and depends only on the 
facts that $P_n$ is an orthogonal projection, $V_t$ is self-adjoint and $P_n V_t$ is 
a Hilbert-Schmidt operator.
Using formula \eqref{f1} for the integral kernel of $P_n$ and the fact that $L_n(0)=1$, we get
$$
\Tr(P_n\phi(V_t)P_n)=\Tr(P_n\phi(V_t))= t^p \frac{B}{2\pi} \int_{\R^2} \phi(V(x))dx.
$$
Next, we have
$$
\norm{P_n V_t(I-P_n)}_{\frakS_2}^2
=
\norm{(P_n V_t-V_tP_n)(I-P_n)}_{\frakS_2}^2
\leq
\norm{(P_n V_t-V_tP_n)}_{\frakS_2}^2
=
o(t^p)
$$
as $t\to\infty$ by Lemma~\ref{lma.f1}. Finally, we note that $\phi(0)=0$ and so 
$P_n\phi(P_nV_tP_n)P_n=\phi(P_nV_tP_n)$. Putting this together yields 
the required asymptotics. 
\end{proof}

\begin{proof}[Proof of Proposition~\ref{prp.c2}]
1. Let us first prove \eqref{c3}. 
One has
\begin{multline*}
N((a,\infty); P_m V_t P_n+P_n V_t P_m)
=
N((1,\infty); \frac1a P_m V_t P_n+\frac1a P_n V_t P_m)
\\
\leq 
\frac1{a^2} \norm{P_m V_t P_n+P_n V_t P_m}_{\frakS_2}^2
\leq 
\frac1{a^2} (\norm{P_m V_t P_n}_{\frakS_2}+\norm{P_n V_t P_m}_{\frakS_2})^2
=
\frac{4}{a^2}\norm{P_n V_t P_m}_{\frakS_2}^2
\\
=\frac{4}{a^2}\norm{(P_n V_t -V_t P_n)P_m}_{\frakS_2}^2
\leq 
\frac{4}{a^2}\norm{P_n V_t -V_t P_n}_{\frakS_2}^2
=o(t^p),
\end{multline*}
as $t\to\infty$ by Lemma~\ref{lma.f1}.

2. Let us prove \eqref{c1} and \eqref{c2}.
Let $\phi,\psi\in W_\infty^2(\R)$ be such that $\phi(0)=\psi(0)=0$ 
and 
$$
0\leq \phi(s)\leq \chi_{(a,\infty)}(s)\leq \chi_{[a,\infty)}(s)\leq \psi(s),
\quad s\geq0.
$$
Then 
\begin{align}
N((a,\infty),P_nV_tP_n)&\geq \Tr\phi(P_nV_tP_n),
\label{f4}
\\
N([a,\infty),P_nV_tP_n)&\leq \Tr\psi(P_nV_tP_n),
\label{f5}
\end{align}
The asymptotics of the traces in the r.h.s. of \eqref{f4} and \eqref{f5} 
is given by Lemma~\ref{lma.f2}.
Choosing appropriate sequences of functions $\phi_n$ and $\psi_n$ 
which  converge to $\chi_{(a,\infty)}$ pointwise on $\R\setminus\{a\}$,  
we obtain the required result. 
\end{proof}

\subsection{Eigenvalues below the essential spectrum}
\begin{lemma}\label{lma.c1}
One has
\begin{equation}
\lim_{E\to\infty}\limsup_{t\to\infty}t^{-p}N((-\infty,-E); H_t)=0.
\label{c4}
\end{equation}
\end{lemma}
\begin{proof}
By the diamagnetic inequality (see, e.g., \cite[Section~2]{AHS} and references therein), one 
has the following estimate involving the integral kernels of the resolvents
of $H_0$ and $-\Delta$: 
$$
\abs{(H_0+E)^{-1}[x,y]}\leq (-\Delta+E)^{-1}[x,y].
$$
It follows that the Hilbert-Schmidt norm of $V(H_0+E)^{-1}$ can be estimated
by the Hilbert-Schmidt norm of $V(-\Delta+E)^{-1}$: 
\begin{multline}
\norm{V(H_0+E)^{-1}}_{{\frakS_2}}^2
\leq
\norm{V(-\Delta+E)^{-1}}_{{\frakS_2}}^2
\\
=
\frac1{2\pi}\int_{\R^2}\abs{V(x)}^2dx \int_{\R^2} (p^2+E)^{-2}dp
=
\frac{C}{E}\int_{\R^2}\abs{V(x)}^2dx.
\label{d2a}
\end{multline}
Next, recall the following well known estimate. 
If $L\geq0$ and $M\geq0$ are bounded self-adjoint operators such that
$LM$ is Hilbert-Schmidt, then 
\begin{equation}
\norm{L^{1/2}ML^{1/2}}_{\frakS_2}^2
=
\Tr (L^{1/2}MLML^{1/2})=\Tr(LMLM)\leq \norm{LM}_{\frakS_2}^2.
\label{d2}
\end{equation}
Using the Birman-Schwinger principle in the form \eqref{b24a}
and the estimates \eqref{d2a}, \eqref{d2} (with $L=\abs{V_t}$, $M=(H_0+E)^{-1}$),  we get
\begin{multline*}
N((-\infty,-E),H_t)
\leq
N((-\infty,-E), H_0-\abs{V_t})
=
N((1,\infty); \abs{V_t}^{1/2}(H_0+E)^{-1}\abs{V_t}^{1/2})
\\
\leq 
\norm{\abs{V_t}^{1/2} (H_0+E)^{-1}\abs{V_t}^{1/2}}_{\frakS_2}^2
\leq 
\norm{\abs{V_t} (H_0+E)^{-1}}_{\frakS_2}^2
\\
=
\frac{C}{E}\int_{\R^2}\abs{V_t(x)}^2dx
=
t^p\frac{C}{E}\int_{\R^2}\abs{V(x)}^2dx,
\end{multline*}
which yields the required result.
\end{proof}
\subsection{Proof of Theorem~\ref{th1}}
1. 
We first note the following elementary continuity properties 
of the asymptotic coefficients $A(a,V)$, $A[a,V]$ for any $a>0$: 
\begin{align}
A(a+0,V)&=A(a,V),
\qquad
\qquad \  A[a-0,V]=A[a,V],
\label{d1}
\\
\lim_{\e\to+0}A(a; V-\e \abs{V})&=A(a; V),
\quad
\lim_{\e\to+0}A[a; V+\e \abs{V}]=A[a; V].
\label{d5}
\end{align}

2. 
In order to prove \eqref{a2}, \eqref{a3}, let us apply 
Theorem~\ref{th.e1}. 
Assumption \eqref{e1} is clearly fulfilled; \eqref{e2}, \eqref{e2a} hold true by \eqref{c3} and
\eqref{e3} holds true by Lemma~\ref{lma.c1}. Thus,  Theorem~\ref{th.e1} is applicable.

Let us prove the upper bound \eqref{a2}; the lower bound \eqref{a3} can be 
proven in the same way. 
Denote $W=V+\e\abs{V}$; by \eqref{e4} and \eqref{c1}, \eqref{c2}, we have for all sufficiently small $a>0$: 
\begin{multline*}
\limsup_{t\to\infty} t^{-p} \Xi(\lambda; H_t,H_0)
\leq 
\sum_{n=0}^\infty
\limsup_{t\to\infty} t^{-p} \Xi(\lambda-a; H_0+P_nW_tP_n,H_0)
\\
=
\sum_{n=0}^\infty
\limsup_{t\to\infty} t^{-p} \Xi(\lambda-a-\Lambda_n; P_nW_tP_n,0)
\\
=
\sum_{n=0}^q
\limsup_{t\to\infty} t^{-p} N([\lambda-a-\Lambda_n,\infty); P_nW_tP_n)
-
\sum_{n=q+1}^\infty
\liminf_{t\to\infty} t^{-p} N((\Lambda_n+a-\lambda,\infty); -P_nW_tP_n)
\\
\leq
\sum_{n=0}^q A[\lambda-a-\Lambda_n,W]
-
\sum_{n=q+1}^\infty A(\Lambda_n+a-\lambda, -W)
\\
\leq
\sum_{n=0}^q A[\lambda-a-\Lambda_n,V+\e \abs{V}]
-
\sum_{n=q+1}^N A(\Lambda_n+a-\lambda, -V-\e\abs{V})
\end{multline*}
for any $N\geq q+1$. 
Letting $a\to+0$, $\e\to+0$ and using the continuity properties \eqref{d1}, \eqref{d5}, 
we obtain
$$
\limsup_{t\to\infty} t^{-p} \Xi(\lambda; H_t,H_0)
\leq 
\sum_{n=0}^q A[\lambda-\Lambda_n,V]
-
\sum_{n=q+1}^N A(\Lambda_n-\lambda, -V).
$$
Since $N$ here is arbitrary, we obtain \eqref{a2}.

3. 
Let us prove \eqref{a5}.
Combining \eqref{a2}, \eqref{a3} with identities  \eqref{b20}, \eqref{b21}, 
one obtains 
$$
\limsup_{t\to\infty} t^{-p}N([\lambda_1,\lambda_2); H_t)
\leq
\frac{B}{2\pi}\sum_{n\geq 0} \meas\{x\in\R^2\mid \lambda_1-\Lambda_n\leq V(x)\leq  \lambda_2-\Lambda_n\}.
$$
Replacing $\lambda_2$ by $\lambda_2+\varepsilon$, letting $\varepsilon\to+0$ and using the continuity 
properties \eqref{d1}, one obtains \eqref{a5}. In the same way one obtains the lower bound \eqref{a6}.
\qed

\section*{Acknowledgements}
The author is grateful to  N.~A.~Azamov, A.~L.~Carey, E.~B.~Davies, R.~Hempel,  G.~Rozenblum and A.~Sobolev  for useful 
discussions and hints on the literature. 
The author is particularly grateful to N.~Filonov for a careful critical reading of the manuscript
and making a number of very useful suggestions.

\end{document}